\documentclass[10pt]{amsart}
\usepackage{graphicx, color}
\usepackage{amssymb}
\usepackage{amsmath}
\usepackage[T1]{fontenc}
\usepackage[displaymath, mathlines, running]{lineno}
\newtheorem{theorem}{Theorem}[section]

\newtheorem{corollary}[theorem]{Corollary}
\newtheorem{proposition}[theorem]{Proposition}

\newenvironment{proof of main thm}{{\bf Proof of Theorem \ref{main thm}.}}{\hfill\fbox{}\par\vspace{.2cm}}
\numberwithin{equation}{section}

\def\charf {\mbox{{\text 1}\kern-.24em {\text l}}}

\def\bea{\begin{eqnarray*}}
\def\eea{\end{eqnarray*}}
\def\be{\begin{eqnarray}}
\def\ee{\end{eqnarray}}

\begin{document}

\title[Rigidity of volume-minimizing free boundary hypersurface]{Rigidity of volume-minimizing free boundary hypersurfaces in high dimensional Riemannian manifolds}

%    Remove any unused author tags.

\author{Sanghun Lee}
\address{Department of Mathematics and Institute of Mathematical Science, Pusan National University, Busan 46241, Republic of Korea}
\email{kazauye@pusan.ac.kr}

%    author two information

%\thanks{The author was supported by the National Research Foundation of Korea (NRF) grant funded by the Korea government (No. RS-2023-00214167).}


\date{}

\dedicatory{}

\begin{abstract}
In this paper, we establish a volume estimate and an associated rigidity result for volume-minimizing free boundary hypersurfaces in an $(n+1)$-dimensional Riemannian manifold $M$. First, we derive a volume estimate for stable minimal free boundary hypersurfaces in terms of the scalar curvature of $M$ and the second fundamental form of the boundary $\partial M$. We then prove rigidity results of $M$ in the corresponding equality cases.
\end{abstract}

\maketitle
\section{Introduction}

The interaction between scalar curvature and the geometry of minimal hypersurfaces has played an important role in Riemannian geometry. A classical result of Schoen and Yau \cite{SC3} shows that, in a three-dimensional Riemannian manifold with positive scalar curvature, a closed, two-sided, stable minimal surface must have genus zero. Under the weaker assumption of non-negative scalar curvature, Fischer-Colbrie and Schoen \cite{FC} proved that such a stable minimal surface can only have genus zero or one. Moreover, in the genus one case, the surface is totally geodesic and the normal Ricci curvature of the ambient manifold vanishes along the surface. Motivated by these results, Cai and Galloway \cite{CG} proved that if such a minimal torus is additionally area-minimizing, then a neighborhood of the torus splits isometrically as a product. Analogous results under positive and negative scalar curvature assumptions were established by Bray, Brendle, and Neves \cite{BBN}, and Nunes \cite{NU}, respectively. Micallef and Moraru \cite{MM} later gave a unified approach to these rigidity results. Related high dimensional results can be found in \cite{BCBS, CG2, CG3, DE, ME, MOT, MO, ZH}. In particular, rigidity results for volume-minimizing Einstein hypersurfaces in high dimensional Riemannian manifolds with positive scalar curvature were obtained in \cite{BCBS, DE}.

It is natural to consider analogous questions for Riemannian manifolds with boundary. In this setting, free boundary minimal hypersurfaces arise as critical points of the volume functional among variations preserving the ambient boundary and meet the boundary of the ambient manifold orthogonally. Rigidity for minimal free boundary hypersurfaces have been studied in several settings \cite{AM, BAC, BC, CR}. In particular, Ambrozio \cite{AM} established rigidity results for area-minimizing free boundary surfaces in three-dimensional Riemannian manifolds with mean convex boundary. Barros and Cruz \cite{BC} obtained volume estimates and rigidity results for volume-minimizing free boundary hypersurfaces with non-positive Yamabe constant in high dimensional Riemannian manifold with mean convex boundary.

The purpose of this paper is to study a positive scalar curvature analogue of these rigidity results for high dimensional free boundary hypersurfaces. More precisely, we consider a compact volume-minimizing free boundary hypersurface $\Sigma$ in an
$(n+1)$-dimensional Riemannian manifold $M$ with weakly convex boundary, under the assumptions $Ric^{M}\geq 0$ and $\inf_M R^M>0$. We further assume that the induced metric on $\Sigma$ is Einstein, namely, $Ric^{\Sigma} = \frac{R^{\Sigma}}{n}g_{\Sigma}$. Under these assumptions, we establish an upper bound for the volume of $\Sigma$ and prove a rigidity result in the equality case.

An important ingredient in our result is the recent spectral volume comparison theorem of Li \cite{LI}. Li established a Bishop-Gromov volume comparison for compact Riemannian manifolds with weakly convex boundary under a spectral lower bound involving the smallest eigenvalue of the Ricci curvature. We recall the version that will be used in this paper.

\begin{theorem} [\cite{LI}] \label{vol com thm}
Let $\Sigma$ be an $n$-dimensional compact connected smooth manifold ($n \geq 3$) with weakly convex boundary $\partial\Sigma$, and let $0 \leq \gamma \leq \frac{n-1}{n-2}$, $\lambda > 0$. Assume that there exists a positive smooth function $\Phi \in C^{\infty}(\Sigma)$ satisfying
\begin{align*}
\begin{cases}
\Phi Ric^{\Sigma}_{-} - \gamma\Delta_{\Sigma}\Phi \geq (n-1)\lambda \Phi \,\, {\rm in} \,\, \Sigma, \\
\frac{\partial \Phi}{\partial\nu} = 0, \,\, {\rm on} \,\, \partial\Sigma,
\end{cases}
\end{align*}
where $Ric^{\Sigma}_{-}$ denotes the smallest eigenvalue of the Ricci curvature of $\Sigma$ and $\nu$ is the unit conormal vector of $\partial\Sigma$ in $\Sigma$. If $\Phi$ attains its minimum on $\partial\Sigma$, then
\begin{align} \label{vol com ine}
{\rm Vol}(\Sigma) \leq \lambda^{-\frac{n}{2}}{\rm Vol}(\mathbb{S}^{n}_{+}).
\end{align}
Moreover, if equality holds in (\ref{vol com ine}), then $\Sigma$ is isometric to the round hemisphere of radius $\lambda^{-\frac{1}{2}}$. 
\end{theorem}

Here and throughout the paper, \textit{weakly convex boundary} means that the second fundamental form of the boundary is nonnegative.

Combining Theorem \ref{vol com thm} with the stability inequality for minimal free boundary hypersurfaces, we obtain the following volume estimate and rigidity result.

\begin{theorem} \label{main thm}
Let $M$ be an $(n+1)$-dimensional Riemannian manifold $(n \geq 3)$ with weakly convex boundary $\partial M$. Assume that $Ric^{M} \geq 0$ and $\inf_{M}R^{M} > 0$. Let $\Sigma$ be a compact connected two-sided, properly embedded, volume-minimizing free boundary hypersurface satisfying $Ric^{\Sigma} = \frac{R^{\Sigma}}{n}g_{\Sigma}$. Then
\begin{align} \label{ine vol}
{\rm Vol}(\Sigma)^{\frac{2}{n}}\inf_{M}R^{M} \leq n(n-1){\rm Vol}(\mathbb{S}^{n}_{+})^{\frac{2}{n}},
\end{align}
where $\mathbb{S}^{n}_{+}$ denotes the round hemisphere. Moreover, if equality holds in (\ref{ine vol}), then, up to scaling, $\Sigma$ is isometric to $\mathbb{S}^{n}_{+}$ and a neighborhood of $\Sigma$ in $M$ isometric to $\mathbb{S}^{n}_{+} \times (-\epsilon,\epsilon)$ with the product metric.
\end{theorem}

\section{Preliminaries}

In this section, we recall some basic definitions and facts concerning free boundary hypersurfaces.

Let $M$ be an $(n+1)$-dimensional Riemannian manifold with non-empty boundary $\partial M$, and let $\Sigma \subset M$ be a compact two-sided, properly embedded hypersurface with boundary satisfying $\Sigma \cap \partial M = \partial\Sigma$. We denote by $N$ a unit normal vector field along $\Sigma$, and by $\nu$ the outward unit conormal to $\partial\Sigma$ in $\Sigma$. Let $\bar{N}$ denote the outward unit normal vector to $\partial M$. We say that $\Sigma$ is a \textit{free boundary hypersurface} if it meets $\partial M$ orthogonally along $\partial\Sigma$. After choosing the orientations consistently, $\nu = \bar{N}$ along $\partial\Sigma$. Let $\Psi: \Sigma \times (-\epsilon,\epsilon) \rightarrow M$ be a smooth \textit{variation} of $\Sigma$, with $\Psi(x,0) = x$, such that $\Psi(\partial\Sigma, t) \subset \partial M$ for every $t \in (-\epsilon,\epsilon)$. We write $\Psi(\Sigma,t) := \Sigma_{t}$ and denote the variational vector field by $X = \left.\frac{\partial\Psi}{\partial t}\right|_{t=0}$. Let $\varphi=\langle X,N\rangle$ be the normal component of the variational vector field. The volume functional associated with the variation is defined by
\begin{align*}
V(t)={\rm Vol}(\Sigma_t) = \int_{\Sigma}dv_{t},
\end{align*}
where $dv_{t}$ denotes the volume element induced on $\Sigma_{t}$. Then the first variation formula is 
\begin{align} \label{first vari} 
V'(0) = \int_{\Sigma} H\varphi \, dv + \int_{\partial\Sigma}\langle X, \nu \rangle \, ds, 
\end{align} 
where $H$ denotes the mean curvature of $\Sigma$. If $\Sigma$ is a free boundary hypersurface, then $\nu=\bar{N}$ along $\partial\Sigma$, while $X\in T(\partial M)$. Hence $\langle X,\nu\rangle=0$ along $\partial\Sigma$. Consequently, $\Sigma$ is a \emph{minimal free boundary hypersurface} if and only if $V'(0) = 0$ for every admissible variation. We next recall the second variation formula. Let $h^{\Sigma}$ denote the second fundamental form of $\Sigma$ in $M$, and let
$h^{\partial M}$ denote the second fundamental form of $\partial M$ with respect to $\bar{N}$. For a minimal free boundary hypersurface, the second variation is given by
\begin{align} \label{second vari}
\begin{aligned}
V''(0) &= \int_{\Sigma} \left( \vert\nabla^{\Sigma}\varphi\vert^{2} - \left(Ric^{M}(N,N) + \vert h^{\Sigma} \vert ^{2} \right)\varphi^{2} \right)\, dv \\
&\quad - \int_{\partial\Sigma} h^{\partial M}(N,N)\varphi^{2}\,ds.
\end{aligned}
\end{align}
A minimal free boundary hypersurface $\Sigma$ is \textit{stable} if $V''(0) \geq 0$ for every smooth admissible variation. By the second variation formula, this is equivalent to
\begin{align*}
\begin{aligned}
0 &\leq \int_{\Sigma} \left(\vert \nabla^{\Sigma}\varphi \vert^{2} - \left(Ric^{M}(N,N) + \vert h^{\Sigma} \vert^{2}\right)\varphi^{2}\right) \, dv \\
&\quad -\int_{\partial\Sigma} h^{\partial M}(N,N)\varphi^{2} \, ds
\end{aligned}
\end{align*}
for every $\varphi \in C^{\infty}(\Sigma)$.

\section{Volume estimates and CMC foliations}

In this section, we estimate the volume of stable minimal free boundary hypersurface in $(n+1)$-dimensional Riemannian manifold with weakly convex boundary. Moreover, we construct the CMC (constant mean curvature) foliations in a neighborhood of stable minimal free boundary hypersurface.

\begin{proposition} \label{prop 3.1}
Let $M$ be an $(n+1)$-dimensional Riemannian manifold with weakly convex boundary $\partial M$. Assume that $Ric^{M} \geq 0$ and $\inf_{M}R^{M} > 0$. Let $\Sigma$ be a compact connected two-sided, properly embedded, stable minimal free boundary hypersurface satisfying $Ric^{\Sigma} = \frac{R^{\Sigma}}{n}g_{\Sigma}$. Then
\begin{align} \label{ine 3.1}
{\rm Vol}(\Sigma)^{\frac{2}{n}}\inf_{M}R^{M} \leq n(n-1){\rm Vol}(\mathbb{S}^{n}_{+})^{\frac{2}{n}},
\end{align}
where $\mathbb{S}^{n}_{+}$ denotes the unit round hemisphere. Moreover, if equality holds in (\ref{ine 3.1}), then, up to scaling, $\Sigma$ is isometric to $\mathbb{S}^{n}_{+}$.
\end{proposition}

\begin{proof}
Since $\Sigma$ is a stable minimal free boundary hypersurface, we have
\begin{align*}
\begin{aligned}
0 \leq& \int_{\Sigma}\vert \nabla^{\Sigma}\varphi \vert^{2} dv - \int_{\Sigma}\varphi^{2}\left(Ric^{M}(N,N) + \vert h^{\Sigma} \vert^{2}\right) dv \\
&- \int_{\partial\Sigma}h^{\partial M}(N,N)\varphi^{2}ds
\end{aligned}
\end{align*}
for every $\varphi \in C^{\infty}(\Sigma)$. Taking $\varphi = 1$, we obtain
\begin{align*}
\begin{aligned}
0 \leq& -\int_{\Sigma}\left(Ric^{M}(N,N) + \vert h^{\Sigma} \vert^{2}\right)dv - \int_{\partial\Sigma} h^{\partial M}(N,N) ds \\
\leq& - \int_{\Sigma}\vert h^{\Sigma} \vert^{2} dv \leq 0,
\end{aligned}
\end{align*}
where we used $Ric^{M} \geq 0$ and the weakly convexity of $\partial M$. Hence, $Ric^{M}(N,N) = 0$ and $h^{\Sigma} = 0$ on $\Sigma$, and $h^{\partial M} (N,N) = 0$ along $\partial\Sigma$. Moreover, since $\Sigma$ meets $\partial M$ orthogonally, we have
\begin{align*}
h^{\partial\Sigma}(V,W) = h^{\partial M}(V,W)
\end{align*}
for all $V,W \in T(\partial\Sigma)$. This implies that 
\begin{align} \label{ine 3.2}
h^{\partial\Sigma} \geq 0 \quad \text{along } \partial\Sigma. 
\end{align}
On the other hand, the Gauss equation gives
\begin{align*}
Ric^{M}(N,N) = \frac{1}{2}\left(R^{M} - R^{\Sigma} + H^{2} - \vert h^{\Sigma} \vert^{2}\right).
\end{align*}
Since $\Sigma$ is minimal, $Ric^{M}(N,N) = 0$, and $h^{\Sigma} = 0$, we obtain $R^{M} = R^{\Sigma}$ on $\Sigma$. Therefore, using $Ric^{\Sigma} = \frac{R^{\Sigma}}{n}g_{\Sigma}$, the smallest eigenvalue $Ric^{\Sigma}_{-}$ of the Ricci curvature of $\Sigma$ satisfies
\begin{align*}
Ric^{\Sigma}_{-} = \frac{R^{\Sigma}}{n} = \frac{R^{M}}{n} \geq \frac{\inf_{M}R^{M}}{n}.
\end{align*}
Set $\lambda:= \frac{\inf_{M}R^{M}}{n(n-1)} > 0$. Then 
\begin{align*}
Ric^{\Sigma}_{-} \geq (n-1)\lambda.
\end{align*}
Taking $\Phi = 1$ and $\gamma = 0$, we have
\begin{align*}
\Phi Ric^{\Sigma}_{-} - \gamma\Delta_{\Sigma}\Phi = Ric^{\Sigma}_{-} \geq (n-1)\lambda\Phi
\end{align*}
in $\Sigma$, and
\begin{align*}
\frac{\partial\Phi}{\partial\nu} = 0
\end{align*}
along $\partial\Sigma$. Moreover, since $\Phi = 1$, it attains its minimum on $\partial\Sigma$. From (\ref{ine 3.2}), all the hypotheses of Theorem \ref{vol com thm} are satisfied. Hence,
\begin{align*}
{\rm Vol}(\Sigma) \leq \lambda^{-\frac{n}{2}}{\rm Vol}(\mathbb{S}^{n}_{+}).
\end{align*}
Equivalently,
\begin{align*}
\begin{aligned}
{\rm Vol}(\Sigma)^{\frac{2}{n}} &\leq \lambda^{-1}{\rm Vol}(\mathbb{S}^{n}_{+})^{\frac{2}{n}} \\
&= \frac{n(n-1)}{\inf_{M}R^{M}}{\rm Vol}(\mathbb{S}^{n}_{+})^{\frac{2}{n}},
\end{aligned}
\end{align*}
which proves (\ref{ine 3.1}). Finally, suppose that equality holds in (\ref{ine 3.1}). Then equality holds in (\ref{vol com ine}) of Theorem \ref{vol com thm}. By the rigidity statement in Theorem \ref{vol com thm}, $\Sigma$ is isometric to the round hemisphere of radius $\lambda^{-\frac{1}{2}}$. Therefore, up to scaling, $\Sigma$ is isometric to $\mathbb{S}^{n}_{+}$.

\end{proof}

The following proposition indicates that it is possible to construct a CMC foliation of a minimal free boundary hypersurface in an $(n+1)$-dimensional Riemannian manifold with certain properties. For further details, refer to \cite{AM, BC}.

Let smooth vector fields $Y$ in $M$ that coincides with $N$ on $\Sigma$ and $Y(p)$ is tangent to $\partial M$ for all $p \in \partial M$. We denote $\phi = \phi(x,t)$ the flow of $Y$.

\begin{proposition} [\cite{AM, BC, CR}] \label{foliation}
Let $M$ be an $(n+1)$-dimensional Riemannian manifold with boundary $\partial M$. Let $\Sigma$ be a compact connected two-sided, properly embedded, free boundary hypersurface. If $\Sigma$ is totally geodesic, $Ric^{M}(N,N) = 0$ on $\Sigma$, and $h^{\partial M}(N,N) = 0$ along $\partial\Sigma$, then there exists $\epsilon > 0$ and a smooth function $w: \Sigma \times (-\epsilon,\epsilon) \rightarrow \mathbb{R}$ such that, for every $t \in (-\epsilon,\epsilon)$, the set $\Sigma_{t} = \left\{\phi(x,w(x,t)) \vert x \in \Sigma \right\}$ is a free boundary hypersurface with constant mean curvature $H(t)$. Moreover,
\begin{align*}
w(x,0) = 0, \,\, \int_{\Sigma}\left( w(x,t) - t \right) dv = 0, \,\, {\rm and} \,\, \frac{\partial w}{\partial t}(x,0) = 1 \,\,.
\end{align*}
In particular, for some smaller $\epsilon$, the family $\{\Sigma_{t}\}_{t \in (-\epsilon,\epsilon)}$ forms a foliation of a neighborhood of $\Sigma$ in $M$.
\end{proposition}

\section{Proof of Theorem \ref{main thm}}

In this section, we prove the rigidity of an $(n+1)$-dimensional Riemannian manifold with weakly convex boundary $\partial M$.

\begin{proof of main thm}
The rigidity of Proposition \ref{prop 3.1} implies that, up to scaling, $\Sigma$ is isometric to the round hemisphere $\mathbb{S}^{n}_{+}$. The same proposition also yields $h^{\Sigma} = 0$, $Ric^{M}(N,N) = 0$ on $\Sigma$, and $h^{\partial M}(N,N) = 0$ along $\partial\Sigma$. Hence, Proposition \ref{foliation} provides a free boundary CMC foliation $\{\Sigma_{t}\}_{t\in (-\epsilon,\epsilon)}$ of a neighborhood of $\Sigma$. Let $G: \Sigma \times (-\epsilon,\epsilon) \rightarrow M$ given by $G_{t}(x) = \phi(x,w(x,t))$, be a parametrization of this foliation. Denote by $N_{t}$ a unit normal vector field along $\Sigma_{t}$, and define the \textit{lapse function} by $\rho_{t} = \langle \partial_{t}G, N_{t} \rangle$. Since $\rho_{0} = 1$ and $\Sigma$ is compact, we may assume that $\rho_{t} > 0$ for every $t \in (-\epsilon,\epsilon)$. It is well known that
\begin{align} \label{eq 4.1}
\begin{cases}
H'(t) = -\Delta_{\Sigma_{t}}\rho_{t} - \left(Ric^{M}(N_{t},N_{t}) + \vert h^{\Sigma_{t}} \vert^{2}\right)\rho_{t} \quad &\text{on } \Sigma_{t} ,\\
\frac{\partial\rho_{t}}{\partial\nu_{t}} = h^{\partial M}(N_{t},N_{t})\rho_{t} \quad &\text {along } \partial\Sigma_{t},
\end{cases}
\end{align}
where $\nu_{t}$ denotes unit conormal of $\partial\Sigma_{t}$ in $\Sigma_{t}$ (see \cite{AM, BC}). Multiplying (\ref{eq 4.1}) by $\rho^{-1}_{t}$, we obtain
\begin{align} \label{ine 4.2}
\begin{aligned}
H'(t)\rho^{-1}_{t} &= -\rho^{-1}_{t}\Delta_{\Sigma_{t}}\rho_{t} - \left(Ric^{M}(N_{t},N_{t}) + \vert h^{\Sigma_{t}} \vert^{2}\right) \\
&\leq -\rho^{-1}_{t}\Delta_{\Sigma_{t}}\rho_{t},
\end{aligned}
\end{align}
where we used $Ric^{M} \geq 0$. Integrating (\ref{ine 4.2}) over $\Sigma_{t}$, we obtain
\begin{align} \label{eq 4.3}
\begin{aligned}
H'(t)\int_{\Sigma_{t}}\rho^{-1}_{t} dv &\leq - \int_{\Sigma_{t}} \rho^{-1}_{t}\Delta_{\Sigma_{t}}\rho_{t} \, dv \\
&= -\int_{\partial\Sigma_{t}}\frac{\partial\rho_{t}}{\partial\nu_{t}}\frac{1}{\rho_{t}} \, ds - \int_{\Sigma_{t}}\frac{\vert \nabla^{\Sigma_{t}}\rho_{t} \vert^{2}} {\rho^{2}_{t}} \, dv \\
&= -\int_{\partial\Sigma_{t}}h^{\partial M}(N_{t},N_{t}) \, ds - \int_{\Sigma_{t}}\frac{\vert \nabla^{\Sigma_{t}}\rho_{t} \vert^{2}}{\rho^{2}_{t}} \, dv \\
&\leq -\int_{\Sigma_{t}} \frac{\vert \nabla^{\Sigma_{t}}\rho_{t} \vert^{2}}{\rho^{2}_{t}} \, dv \leq 0,
\end{aligned}
\end{align}
where in the last inequality we have used the weak convexity of $\partial M$. Since $\rho_{t} > 0$, it follows from (\ref{eq 4.3}) that $H'(t) \leq 0$. Since $H(0) = 0$, we obtain
\begin{align} \label{ine 4.3}
H(t) \leq 0 \leq H(-t) \quad \text{for all } t \in [0,\epsilon).
\end{align}
On the other hand, since each $\Sigma_{t}$ is a free boundary hypersurface, the first variation formula of volume gives
\begin{align} \label{ine 4.4}
\frac{d}{dt}{\rm Vol}(\Sigma_{t}) = H(t)\int_{\Sigma_{t}}\rho_{t} \, dv.
\end{align}
Since $\rho_{t} > 0$, it follows from (\ref{ine 4.3}) and (\ref{ine 4.4}) that
\begin{align*}
{\rm Vol}(\Sigma_{t}) \leq {\rm Vol}(\Sigma) \quad \text{for all } t \in (-\epsilon,\epsilon).
\end{align*}
Since $\Sigma$ is volume-minimizing, we have ${\rm Vol}(\Sigma_{t}) = {\rm Vol}(\Sigma)$ for every $t \in (-\epsilon,\epsilon)$. It follows from (\ref{ine 4.4}), together with $\rho_{t} > 0$, that $H(t) = 0$ for all $t \in (-\epsilon,\epsilon)$. Consequently, $H'(t) = 0$. Returning to (\ref{eq 4.3}), we obtain
\begin{align*}
0 \leq -\int_{\partial\Sigma_{t}}h^{\partial M}(N_{t},N_{t}) \, ds - \int_{\Sigma_{t}}\frac{\vert \nabla^{\Sigma_{t}}\rho_{t} \vert^{2}}{\rho^{2}_{t}} \, dv \leq 0.
\end{align*}
Therefore, $h^{\partial M}(N_{t},N_{t}) = 0$ along $\partial\Sigma_{t}$, and $\nabla^{\Sigma_{t}}\rho_{t} = 0$ on $\Sigma_{t}$. Thus $\rho_{t}$ is constant on each $\Sigma_{t}$. Hence, (\ref{eq 4.1}) reduces to
\begin{align*}
 0 = -\left(Ric^{M}(N_{t},N_{t}) + \vert h^{\Sigma_{t}} \vert^{2}\right)\rho_{t}.
\end{align*}
Since $\rho_{t} > 0$, $Ric^{M} \geq 0$, and $\vert h^{\Sigma_{t}} \vert^{2} \geq 0$, we conclude that $Ric^{M}(N_{t},N_{t}) = 0$ and $h^{\Sigma_{t}} = 0$ on every $\Sigma_{t}$. We next show that the unit normal vector field $N_{t}$ is parallel along $\Sigma_{t}$. Let $(x_{1}, \cdots, x_{n})$ be local coordinates on $\Sigma$. Since $\Sigma_{t}$ is totally geodesic, $\nabla_{\partial_{x_{i}}G}N_{t} = 0$. Furthermore, since $\vert N_{t} \vert = 1$, 
\begin{align*}
\left\langle \nabla_{\partial_{t}G}N_{t}, N_{t} \right\rangle = 0,
\end{align*}
and hence $\nabla_{\partial_{t}G}N_{t}$ is tangent to $\Sigma_{t}$. Since $\rho_{t}$ is constant on $\Sigma_{t}$, we have
\begin{align*}
\begin{aligned}
0 = \partial_{x_{i}}\rho_{t} &= \left\langle \nabla_{\partial_{x_{i}}G}N_{t}, \partial_{t}G \right\rangle + \left\langle N_{t}, \nabla_{\partial_{x_{i}}G}\partial_{t}G \right\rangle \\
&= \left\langle N_{t}, \nabla_{\partial_{t}G} \partial_{x_{i}}G \right\rangle \\
&= - \left\langle \nabla_{\partial_{t}G}N_{t}, \partial_{x_{i}}G \right\rangle
\end{aligned}
\end{align*}
(see also \cite{AM, BC, CR}). Thus $N_{t}$ is a parallel vector field. It follows that the integral curves of $N_{t}$ are geodesic. Consider the exponential map $E(t,x) := {\rm exp}_{x}\left(tN(x)\right)$ for $x \in \Sigma$. Since $N_{t}$ is parallel, its flow is given by the exponential map $E$, which is an isometry for all $t \in (-\epsilon,\epsilon)$. Therefore, we conclude that, up to scaling, a neighborhood of $\Sigma$ in $M$ is isometric to $\mathbb{S}^{n}_{+} \times (-\epsilon,\epsilon)$. 
\end{proof of main thm}

The following corollary extends our local rigidity result to a global rigidity result.

\begin{corollary}
Let $M$ be a connected, complete $(n+1)$-dimensional Riemannian manifold $(n \geq 3)$ with weakly convex boundary $\partial M$. Assume that $Ric^{M} \geq 0$ and $\inf_{M}R^{M} > 0$. Let $\Sigma$ be a compact connected two-sided, properly embedded, free boundary hypersurface satisfying $Ric^{\Sigma} = \frac{R^{\Sigma}}{n}g_{\Sigma}$. Assume, in addition, that $\Sigma$ is volume-minimizing in its isotopy class. If equality holds in (\ref{ine vol}), then, up to scaling, the universal cover $\tilde{M}$ of $M$ is isometric to $\mathbb{S}^{n}_{+} \times \mathbb{R}$.
\end{corollary}

\begin{proof}
Since $\Sigma$ is volume-minimizing in its isotopy class, it is locally volume-minimizing. By the Proposition \ref{prop 3.1} and the assumed equality of (\ref{ine vol}), up to scaling, $\Sigma$ is isometric to $\mathbb{S}^{n}_{+}$. Applying the global conclusion of Theorem 4.2 with constant density in \cite{CR}, we obtain a Riemannian covering $G: \Sigma \times \mathbb{R} \rightarrow M$. Since $\Sigma$ is a round hemisphere, $\Sigma \times \mathbb{R}$ is simply connected. Thus $G$ is the universal covering of $M$. Consequently, up to scaling, the universal cover of $M$ is isometric to $\mathbb{S}^{n}_{+} \times \mathbb{R}$ with the product metric.
\end{proof}


\begin{thebibliography}{00}

\bibitem{AM} L. Ambrozio,
{\it Rigidity of area-minimizing free boundary surfaces in mean convex three-manifolds}, J. Geom. Anal., {\bf 25} (2015), 1001-1017.

\bibitem{BAC} E. Barbosa and F. Conrado,
{\it Discs area-minimizing in mean convex Riemannian $n$-manifolds}, Proc. R. Soc. Edinb. A, {\bf 152} (2022), 1361-1382.

\bibitem{BCBS} A. Barros, C. Cruz, R. Batista, and P. Sousa,
{\it Rigidity in dimension four of area-minimising Einstein manifolds}, Math. Proc. Camb. Phil. Soc., {\bf 158} (2015), 355-363.

\bibitem{BC} A. Barros and C. Cruz,
{\it Free boundary hypersurfaces with non-positive Yamabe invariant in mean convex manifolds}, J. Geom. Anal., {\bf 30} (2020), 3542-3562.

\bibitem{BBN} H. Bray, S. Brendle, and A. Neves,
{\it Rigidity of area-minimizing two-spheres in three-manifolds}, Commun. Anal. Geom., {\bf 18} (2010), 821-830.

\bibitem{CG} M. Cai and G. Galloway,
{\it Rigidity of area-minimizing tori in 3-manifolds of nonnegative scalar curvature}, Commun. Anal. Geom., {\bf 8} (2000), 565-573.

\bibitem{CG2} M. Cai and G. Galloway,
{\it On the topology and area of higher-dimensional black holes}, Class. Quantum Gravity, {\bf 18} (2001), 2707-2718.

\bibitem{CG3} M. Cai, 
{\it Volume Minimizing Hypersurfaces in Manifolds of Nonnegative Scalar Curvature, in Minimal surfaces, geometric analysis and symplectic geometry} (Baltimore, MD, 1999) Adv. Stud. Pure Math. {\bf 34} (2002), no. 1, 1-7.

\bibitem{CR} K. Castro and C. Rosales,
{\it Free boundary stable hypersurfaces in manifolds with density and rigidity results}, J. Geom. Phys., {\bf 79} (2014), 14-28.

\bibitem{DE} H. Deng,
{\it A Bray-Brendle-Neves type inequality for a Riemannian manifold}, Acta Math. Sci., {\bf 41} (2021), 487-492.


\bibitem{FC} D. Fischer-Colbrie, R. Schoen, {\it The structure of complete stable minimal surfaces in 3-manifolds of non-negative scalar curvature}, Comm. Pure App. Math. {\bf 33} (1980) 199-211.

\bibitem{LI} J. Li,
{\it A spectral volume comparison for manifolds with weakly convex boundary}, Bull. Lond. Math. Soc., {\bf 58} (2026), e70473.

\bibitem{ME} A. Mendes,
{\it Rigidity of volume-minimising hypersurfaces in Riemannian 5-manifolds}, Math. Proc. Camb. Phil. Soc., {\bf 167} (2019), 345-353.

\bibitem{MM} M. Micallef and V. Moraru,
{\it Splitting of 3-manifolds and rigidity of area-minimizing surfaces}, Proc. Amer. Math. Soc., {\bf 143} (2015), 2865-2872.

\bibitem{MOT} V. Moraru,
{\it On Area comparison and rigidity involving the scalar curvature}, Ph. D. Thesis, University of Warwick (2013).

\bibitem{MO} V. Moraru,
{\it On area comparison and rigidity involving the scalar curvature}, J. Geom. Anal., {\bf 26} (2016), 294-312.

\bibitem{NU} I. Nunes,
{\it Rigidity of area-minimizing hyperbolic surfaces in three-manifolds}, J. Geom. Anal., {\bf 23} (2013), 1290-1302.

\bibitem{SC3} R. Schoen and S. T. Yau,
{\it Existence of incompressible minimal surfaces and the topology of three dimensional manifolds with non-negative scalar curvature}, Ann. Math., {\bf 110} (1979), 127-142.

\bibitem{ZH} J. Zhu,
{\it Rigidity of area-minimizing $2$-spheres in $n$-manifolds with positive scalar curvature}, Proc. Amer. Math. Soc. {\bf 148} (2020), 3479-3489.



\end{thebibliography}
\end{document}